\documentclass{article}

\usepackage{amsmath, amsthm, amssymb, mathtools, hyperref}
\usepackage{graphicx}
\usepackage{comment}
\usepackage{bbm}
\usepackage{enumerate}
\usepackage[all]{xy}
\usepackage[font=small,labelfont=bf]{caption}
\usepackage{tikz}
\usepackage{color}
\usepackage{booktabs}
\usepackage{algorithm,algpseudocode}

\newtheorem{theorem}{Theorem}
\numberwithin{theorem}{section}
\newtheorem{proposition}[theorem]{Proposition}
\newtheorem{lemma}[theorem]{Lemma}
\newtheorem{corollary}[theorem]{Corollary}
\newtheorem{definition}[theorem]{Definition}
\newtheorem{remark}[theorem]{Remark}
\newtheorem{conjecture}[theorem]{Conjecture}

\newtheorem{example}[theorem]{Example}

\newcommand{\R}{\mathbb{R}}

\newcommand{\Z}{\mathbb{Z}}

\DeclareMathOperator{\ord}{ord}

\newcommand{\tabincell}[2]{\begin{tabular}{@{}#1@{}}#2\end{tabular}}

\begin{document}
	
\title{Computing Linear Systems on Metric Graphs}
\date{}
\author{Bo Lin}
\maketitle

\begin{abstract}
	The linear system $|D|$ of a divisor $D$ on a metric graph has the structure of a cell complex. We introduce the anchor divisors and anchor cells in it - they serve as the landmarks for us to compute the f-vector of the complex and find all cells in the complex. A linear system can also be identified as a tropical convex hull of rational functions. We compute its extremal generators using the landmarks. We apply these methods to some examples - namely the canonical linear systems of some small trivalent graphs.
\end{abstract}

\section{Introduction}\label{sec:int}

In algebraic geometry, the linear systems of divisors on curves are well studied (cf. \cite[\textsection 3]{GH}). Other authors studied linear systems of divisors on metric graphs, for example \cite{Ba08,BF,HMY}. Baker and Norine \cite{BN} proved a graph-theoretic analogue of the Riemann-Roch Theorem, and it was generalized to tropical curves (which may contain unbounded rays) independently by Gathmann and Kerber \cite{GK} and by Mikhalkin and Zharkov \cite{MZ}. The theory of linear systems on metric graphs is applied to algebraic geometry, for example in \cite{BPR}. 

Haase, Musiker and Yu \cite{HMY} studied the cell complex structure of $|D|$ and the tropical semi-module structure of $R(D)$, where $D$ is a divisor on a metric graph. This work is an extension of \cite{HMY}. We focus on the computation of the cell complex $|D|$, namely given a metric graph $\Gamma$ and a divisor $D$ on it, how to find the cells in $|D|$ and the $f$-vector of $|D|$. Since $|D|$ may contain a large number of cells and some of these are complicated, the complexity of computation could be high. We introduce the \emph{anchor cell} that serves as the landmarks to find other cells in $|D|$. As a byproduct we can compute the extremal generators of $R(D)$. We implemented the algorithms and computed some examples - namely the canonical linear systems on some trivalent graphs.  

A \emph{metric graph} $\Gamma$ is a connected undirected graph whose edges have positive lengths. A \emph{divisor} $D$ on $\Gamma$ is a formal finite $\mathbb{Z}$-linear combination $D=\sum_{x\in \Gamma}{D(x)\cdot x}$ of points $X$ in the edges of $\Gamma$. The divisor is \emph{effective} if $D(x)\ge 0$ for all $x\in \Gamma$. The \emph{degree} of a divisor $D$ is $\sum_{x\in \Gamma}{D(x)}$. The \emph{support} of a divisor $D$ on $\Gamma$ is the set $\{x\in \Gamma|D(x)\ne 0\}$, denoted as $supp(D)$.

A \emph{(tropical) rational function} $f$ on $\Gamma$ is a continuous function $f:\Gamma \to \R$ that is piecewise-linear on each edge of $\Gamma$ with finitely many pieces and integer slopes. The \emph{order} $\ord_{x}(f)$ of $f$ at a point $x\in \Gamma$ is the sum of the outgoing slopes at $x$ along all directions. Note that if $x$ is an interior point of a linear piece of $f$, then there are two directions at $x$, and $x$ has two opposite outgoing slopes, so $\ord_{x}(f)=0$. The \emph{principal divisor} associated to $f$ is 
\[(f)=\mathop{\sum}_{x\in \Gamma}{\ord_{x}(f)\cdot x}.\] 
So the support of $(f)$ is always finite.

Two divisors $D$ and $D'$ are \emph{linearly equivalent} if $D-D'=(f)$ for some rational function $f$, denoted as $D\sim D'$. For any divisor $D$ on $\Gamma$, let $R(D)$ be the set of all rational functions $f$ on $\Gamma$ such that the divisor $D+(f)$ is effective, and $|D|=\{D+(f)|f\in R(D)\}$, the \emph{linear system} of $D$. 

\begin{example}\label{exa:C4}
	Let $\Gamma$ be a metric graph with graph-theoretic type $C_4$ and equal edge lengths. Below are examples of $D$ and $f\in R(D)$.
	\begin{figure}[H]
		\centering
		\begin{minipage}[t]{0.4\textwidth}
			\centering
			\begin{tikzpicture}[scale=1]
			\draw (0,0) -- (2,0) -- (2,2) --(0,2) -- (0,0);
			\filldraw [black] (0,0) circle (1pt);
			\filldraw [black] (2,0) circle (1pt);
			\filldraw [black] (2,2) circle (1pt);
			\filldraw [black] (0,2) circle (1pt);
			\node [left] at (0,0) {Q};
			\node [right] at (2,0) {R};
			\node [left] at (0,2) {P};
			\node [right] at (2,2) {S};		
			\end{tikzpicture}
			\caption{the metric graph $\Gamma$}\label{fig:graph}
		\end{minipage}\hspace{2cm}
		\begin{minipage}[t]{0.4\textwidth}
			\centering
			\begin{tikzpicture}[scale=1]
			\draw (0,0) -- (2,0) -- (2,2) --(0,2) -- (0,0);
			\filldraw [black] (0,0) circle (1pt);
			\filldraw [black] (2,0) circle (1pt);
			\filldraw [black] (2,2) circle (1pt);
			\filldraw [black] (0,2) circle (1pt);
			\node [left] at (0,0) {2};
			\node [right] at (2,0) {2};
			\node [left] at (0,2) {2};
			\node [right] at (2,2) {2};		
			\end{tikzpicture}
			\caption{the divisor $D$}\label{fig:D}
		\end{minipage}
		
		\begin{minipage}[t]{0.4\textwidth}
		\centering
            \includegraphics[scale=0.6]{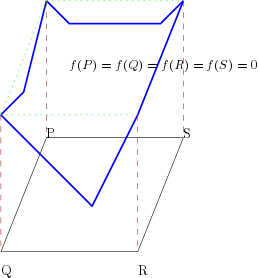}
			\caption{rational function $f$ (blue) on the metric graph $\Gamma$ (black)}\label{fig:f}
		\end{minipage}
		\hspace{2cm}
		\begin{minipage}[t]{0.3\textwidth}
			\centering
			\begin{tikzpicture}[scale=1]
			\draw (0,0) -- (2,0) -- (2,2) --(0,2) -- (0,0);
			\filldraw [black] (2,2) circle (1pt);
			\filldraw [black] (0.5,2) circle (1pt);
			\filldraw [black] (1.5,2) circle (1pt);
			\filldraw [black] (1.33,0) circle (1pt);
			\filldraw [black] (0,1) circle (1pt);
			
			\node [above] at (2,2) {1};
			\node [above] at (0.5,2) {1};
			\node [above] at (1.5,2) {1};
			\node [left] at (0,1) {2};
			\node [above] at (1.33,0) {3};
			\node [left] at (0,0) {Q};
			\node [right] at (2,0) {R};
			\node [left] at (0,2) {P};
			\node [right] at (2,2) {S};
			\end{tikzpicture}
			\caption{the effective divisor $D+(f)$}\label{fig:divi}
		\end{minipage}
	\end{figure}
\end{example}

The metric graph $\Gamma$ is determined by its graph-theoretic type and the lengths of its edges. The graph-theoretic type of $\Gamma$ is called the \emph{skeleton} of $\Gamma$, and the lengths of edges of $\Gamma$ are called the \emph{metric} of $\Gamma$ and denoted by $M$. For an edge $e$ of $\Gamma$ we denote by $M_{e}$ the length of $e$ in $M$. For any metric graph $\Gamma$, the \emph{canonical divisor} $K$ of $\Gamma$ is the divisor on $\Gamma$ with $K(x)=degree(x)-2$ when $x$ is a vertex of $\Gamma$ and $K(x)=0$ otherwise. When we fix the skeleton of $\Gamma$, for any metric $M$ and divisor $D$ on $\Gamma$ we denote by $R(D)_{M}$ the set $R(D)$ and by $|D|_{M}$ the linear system $|D|$.

\begin{remark}
	Given a metric graph $\Gamma=(V,E,M)$ and a divisor $D$ on $\Gamma$. $D$ is \emph{vertex-supported} if $supp(D)\subseteq V$. Note that $supp(D)$ is always a finite set. If $D$ is not vertex-supported, then we can refine $\Gamma$ to get a new metric graph whose set of vertices is $supp(D)$. We shall assume from now on that $D$ is vertex-supported. Also we may assume that $|D|$ is not empty, so if $D$ is not effective, we can consider an effective divisor $D'\in |D|$. It is obvious that $|D'|=|D|$. We shall assume from now that $D$ is vertex-supported. 
\end{remark}

In Section \ref{sec:|D|}, we present the cell complex structure of $|D|$ and introduce the \emph{anchor cells}. We use them as landmarks to find all cells in the complex and prove a combinatorial formula (Corollary \ref{cor:fvec}) for the $f$-vector of $|D|$. We introduce an algorithm to compute the cells of $|D|$ given $\Gamma$ and $D$. In Section \ref{sec:R(D)}, we regard $R(D)$ as a tropical semi-module (convex set) and introduce an algorithm to find the extremal generators of $R(D)$, using a result in \cite{HMY} based on the chip-firing technique. In Section \ref{sec:exp} we apply our algorithms to examples of trivalent graphs. Finally we raise some open problems in Section \ref{sec:q}.

\section{The cell complex $|D|$}\label{sec:|D|}

In this section we present the cell complex structure of $|D|$.

\subsection{The cell complex structure}
Note that if $c$ is a constant rational function on $\Gamma$, then for any rational function $f$ on $\Gamma$, the divisor $D+(f+c)$ is equal to $D+(f)$. Let $\mathbf{1}$ be the set of constant functions on $\Gamma$. The set $R(D)/\mathbf{1}$ can be identified with the linear system $|D|$ by the map $\overline{f}\mapsto D+(f)$. We adapt the formulation of $|D|$ as a cell complex in \cite{HMY}, which originates from \cite{GK} and \cite{MZ}.

\begin{definition}\label{def:cell}
	We identify each open edge $e\in E$ with the interval $(0,M_{e})$ (this implicitly gives an orientation of $e$, while it is independent to the cell complex). Then each cell of $|D|$ is characterized by the following data:
	\begin{itemize}
		\item a nonnegative integer $d_v$ for each $v\in V$;
		\item an ordered partition $d_e=\sum_{i=1}^{r_e}{d_e^{i}}$ of positive integers for some edges $e\in E$;
		\item an integer $m_e$ for each $e\in E$.
	\end{itemize}
	
	A divisor $L$ belongs to this cell if and only if the following statements hold:
	\begin{itemize}
		\item $L(v)=d_v$ for each $v\in V$; 
		\item for those $e\in E$ with the partition above, the divisor $L$ on $e$ is expressed as $\sum_{i=1}^{r_e}{d_e^{i} x_i}$, where $0<x_1<x_2<\ldots<x_{r_e}<M_{e}$; for other edges $e$, $L(x)=0$ for all points $x$ in the interior of $e$; 
		\item for any $f\in R(D)$ such that $L=D+(f)$, the outgoing slope of $f$ at the point $0$ is $m_e$ for each $e\in E$.
	\end{itemize}
\end{definition}

\begin{remark}
	The slopes $m_e$ are also required because two distinct rational functions in $R(D)$ may lead to different effective divisors in $|D|$ with exactly the same values of the $d_v$ and the ordered partitions of $d_e$.
\end{remark}

\begin{example}\label{ext:cou}
	Let $\Gamma=(V,E)$ be a loop $e$ with one vertex $v$ and the length of $e$ equal to $3$. Let $D=3\cdot v$. Then the two effective divisors in Figure \ref{fig:cir} belong to $|D|$. For both, all $d_v$ are $0$ and all $d_e$ are $3$ with one part in the partition, while the slopes $m_e$ are different: one is $-1$ and the other is $-2$. 
\end{example}

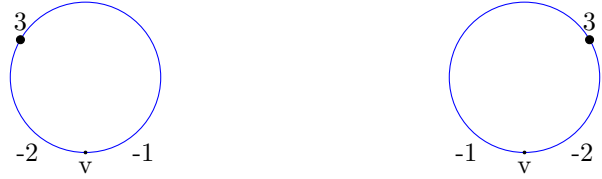
\begin{figure}[h]
	\centering
	\begin{minipage}[t]{0.3\textwidth}
		\centering
		\begin{tikzpicture}[scale=0.5]
		\draw [blue] (2,2) circle [radius=2];
		\filldraw [black] (2,0) circle (1pt);
		\filldraw [black] (0.27,3) circle (3pt);
		\node [below] at (2,0) {v};
		\node [above] at (0.27,3) {3};
		\node [left] at (1,0) {-2};
		\node [right] at (3,0) {-1};		
		\end{tikzpicture}
	\end{minipage}
	\hspace{2cm}
	\begin{minipage}[t]{0.3\textwidth}
		\centering
		\begin{tikzpicture}[scale=0.5]
		\draw [blue] (2,2) circle [radius=2];
		\filldraw [black] (2,0) circle (1pt);
		\filldraw [black] (3.73,3) circle (3pt);
		\node [below] at (2,0) {v};
		\node [above] at (3.73,3) {3};
		\node [left] at (1,0) {-1};
		\node [right] at (3,0) {-2};		
		\end{tikzpicture}
	\end{minipage}
	\caption{Two cells in $|D|$ only differ by $m_e$}\label{fig:cir}
\end{figure}

A cell is represented by any divisor $D+(f)$ in it, because once we have the rational function $f\in R(D)$, we can compute all data in Definition \ref{def:cell}. There is a natural question: how can we tell the dimension of a cell from a representative? The following proposition is a good answer.

\begin{proposition}(\cite[Proposition 13]{HMY})\label{prop:dim}
	Let $D$ be a vertex-supported effective divisor on a metric graph $\Gamma$ and $V$ be the set of vertices in $\Gamma$. Let $C$ be a cell in $|D|$ and $L$ a representative of $C$. Let $I_L=\{x\in \Gamma-V|L(x)>0\}$. Then $\dim C$ is one less than the number of connected components in the graph $\Gamma-I_L$.
\end{proposition}

\begin{corollary}\label{cor:dim}
	Let $D$ be a vertex-supported effective divisor on a metric graph $\Gamma$ and $d=\deg(D)$. Then the dimension of the cell complex $|D|$ is at most $d$. In addition, if $\Gamma$ is $2$-connected, then the dimension of the cell complex $|D|$ is at most $d-1$.
\end{corollary}

\begin{proof}
	Let $C$ be a cell in $|D|$ and $L$ a representative of $C$. Then $\deg(L)=\deg(D)=d$. So the support of $L$ contains at most $d$ points. By definition, $I_L\subseteq supp(L)$, so $I_L$ contains at most $d$ points as well. Now we consider the connected components in the graph $\Gamma-I_L$. First $\Gamma$ is connected. And each time we delete a point in $I_L$, the number of connected components can increase by at most $1$, because the deleted point is always interior to an edge of the current graph. Therefore there are at most $d+1$ connected components in the graph $\Gamma-I_L$. Hence, by Proposition \ref{prop:dim}, we conclude that $\dim C\le d$. If $\Gamma$ is $2$-connected, then when we delete one point in $I_L$ from $\Gamma$, the remaining graph is still connected, so there are at most $d$ connected components in the graph $\Gamma-I_L$. Hence $\dim C\le d-1$.
\end{proof}

In order to find the cells in $|D|$, we need to know whether there are finitely many of them. Fortunately we have the following theorem.

\begin{theorem}(\cite[Theorem 14]{HMY})\label{thm:fin}
	The cell complex $|D|$ has finitely many vertices.
\end{theorem}

Since each cell is uniquely determined by its vertices, we have the following corollary.

\begin{corollary}\label{cor:fin}
	The cell complex $|D|$ has finitely many cells.
\end{corollary}

\subsection{Anchor divisors and anchor cells}

In order to find the cells given $\Gamma$ and $D$, one approach is that we first find the vertices of $|D|$, then find other cells of $|D|$. As for the vertices, it is straightforward to implement the method in the proof of \cite[Theorem 6]{HMY}. However, there might be many general cells in $|D|$ with rather complicated structures, because there might be many parts in the partition on one edge in the data in Definition \ref{def:cell}. To grasp the general cells in $|D|$, we introduce an important type of divisors and cells in $|D|$ that serve as landmarks.

\begin{definition}\label{def:anc}
	A divisor $L$ on $\Gamma$ is an \emph{anchor divisor} if for each edge of $\Gamma$ there is at most one its interior point $x$ with $L(x)>0$. A cell $C$ in $|D|$ is an \emph{anchor cell} if all representatives $L$ of $C$ are anchor divisors.
\end{definition}

\begin{remark}
	The number of anchor divisors in $|D|$ could be infinite. For example if an anchor cell has dimension $1$, then it contains infinitely many divisors and each such divisor is an anchor divisor. Nonetheless, they all represent the same anchor cell because they share the same data in Definition \ref{def:cell}.
\end{remark}

One property of anchor cells is that they contain all vertices of $|D|$.

\begin{lemma}\label{lem:ver}
	Each vertex of $|D|$ is an anchor cell.
\end{lemma}

\begin{proof}
	Suppose that $L\in |D|$ is not an anchor divisor. There exists an edge $e$ of $\Gamma$ and two distinct interior points $P,Q$ of $e$ such that $L(P),L(Q)>0$, then $P,Q\in I_L$. After deleting $P$ and $Q$ the graph is no longer connected. So the number of connected components in the graph $\Gamma-I_L$ is at least $2$. Hence, by Proposition \ref{prop:dim}, $L$ is the representative of a cell with dimension at least $1$. Therefore, if a divisor represents a vertex of $|D|$, it must be an anchor divisor. So each vertex is an anchor cell.
\end{proof}

\begin{definition}\label{def:sa}
	Let $A$ be an anchor cell of $|D|$. For each edge of $\Gamma$ we consider whether there is an ordered partition in the data of $A$. If $e_1,\ldots,e_r$ are the edges of $\Gamma$ such that in the data of $A$ there is an ordered partition $c_i=c_i$ on $e_i$, then we define $s(A)=\sum_{i=1}^{r}{(c_i-1)}$; otherwise we define $s(A)=0$.
\end{definition}

The motivation for anchor cells is the following theorem.

\begin{theorem}(Association to anchor cells)\label{thm:anc}
	Let $D$ be a vertex-supported effective divisor on a metric graph $\Gamma$.
	(1) Suppose $\cal{C}$ is the set of all cells in $|D|$ and $\cal{A}$ is the set of all anchor cells in $|D|$. Then there is a function $a: \cal{C} \to \cal{A}$ such that for $C\in \cal{C}$, let $L\in C$ be a representative of $C$. Then $a(C)$ is represented by an anchor divisor $N$ such that
	\begin{enumerate}[(i)]
		\item $N(v)=L(v)$ for all vertices $v$ of $\Gamma$;
		\item the corresponding rational functions of $N$ and $L$ have the same slopes at the endpoints of every edge of $\Gamma$.
	\end{enumerate}
	
	(2) Let $a$ be the same map as in (1). Then for any anchor cell $A\in \cal{A}$, we have
	\[|a^{-1}(A)|=2^{s(A)}.\]
    Furthermore, for $0\le j\le s(A)$, there are $\binom{s(A)}{j}$ cells of dimension $\dim A +j$ in $a^{-1}(A)$.
\end{theorem}

\begin{proof}
	(1) Given $D$ and $L$, we construct $N$ as follows. Let $f\in R(D)$ such that $L=D+(f)$. Then we construct another rational function $g\in R(D)$. First we let $g(v)=f(v)$ for vertices $v$ of $\Gamma$; next we let $g(x)=f(x)$ for all points $x\in e$, where $e$ is an edge of $\Gamma$ such that $\sum_{y\in e^{\circ}}{L(y)}=0$; for other edges $e$ of $\Gamma$, we regard the open edge $e$ as the interval $(0,l_e)$. Suppose $f$ has $r$ linear pieces $[p_i,p_{i+1}]$ on $e$, with slope $s_i$, where $0=p_0<p_1<\ldots<p_{r}=l_e$. Then for $1\le i\le r-1$, we have $L(p_i)=s_{i+1}-s_{i}$. Since the adjacent linear pieces have different slopes, $s_{i+1}-s_{i}\ne 0$. While $L$ is effective, so $s_{i+1}-s_{i}>0$. Now we let $g$ have two linear pieces on $e$, the one containing $p_0$ with slope $s_1$, the one containing $p_r$ with slope $s_r$. Note that if $r=2$ then $g$ coincides with $f$ on $e$. If $r\ge 3$, since $s_1<s_2<\ldots<s_r$, the graph of the two pieces of $g$ will intersect within the interval, so $g$ is well-defined on $e$. In addition let $p$ be the intersection point of these two pieces, then $(D+(g))(p)=s_r-s_1>0$. We let $N=D+(g)$. By definition of $g$, $N$ is effective and anchor, and $N$ satisfies both (i) and (ii).
	
	(2)	Note that for $C\in \cal{C}$, the two cells $C$ and $a(C)$ share the same first and third parts of their data, and given any edge $e$ of $\Gamma$, either both do not have a partition, or they have the partitions of the same positive integer. The only difference is that the partition of $a(C)$ is always trivial, while the partition of $C$ could be arbitrary. Given a positive integer $t$, there are $2^{t-1}$ ordered partitions of $t$ objects in a row, because each partition corresponds to a $t-1$-tuple of $0,1$, indicating whether or not to break the $t-1$ pairs of adjacent objects. In addition, for $1\le i\le t$, there are $\binom{t-1}{i-1}$ ordered partitions with $i$ parts.
	
	Then if $p_i$ is an interior point on edge $e_i$ with $A(p_i)=c_i$, the partition on $e_i$ of a preimage of $A$ has $2^{c_i-1}$ choices. Once all partitions are determined, so is the preimage. Then $|a^{-1}(A)|=2^{s(A)}$. Furthermore, note that if the number of parts in the partition on one edge increases by $1$, the number of connected components in Proposition \ref{prop:dim} also increases by $1$, which means the dimension of the new cell is one more. Then for $0\le j\le s(A)$, the number of cells with dimension $\dim A +j$ in $a^{-1}(A)$ equals to the number of the $r$-tuples of partitions of $c_1,\ldots,c_r$ with $r+j$ parts in total. This number is $\binom{s(A)}{j}$ by an easy argument of generating functions.
\end{proof}

The following corollary provides a combinatorial formula to compute the $f$-vector of $|D|$ given all of its anchor cells. 

\begin{corollary}\label{cor:fvec}
	Let $D$ be a vertex-supported effective divisor on a metric graph $\Gamma$. For an anchor cell $A$ in $|D|$, define $s(A)$ as in Theorem \ref{thm:anc}. If $A_1,A_2,\ldots,A_m$ are all anchor cells in $|D|$, then for each $d\in \mathbb{N}$ the number of $d$-dimensional cells in the cell complex $|D|$ is the coefficient of $x^d$ in the generating function
	\[\mathop{\sum}_{i=1}^{m}{x^{\dim A_{i}}(1+x)^{s(A_i)}}.\]
\end{corollary}

\begin{proof}
	By Theorem \ref{thm:anc}(1), each cell in $|D|$ is associated to a unique anchor cell of $|D|$, i.e. $\cup_{i=1}^{m}{a^{-1}(A_i)}=\cal{C}$. By Theorem \ref{thm:anc}(2), the generating function for $a^{-1}(A_i)$ is $x^{\dim A_{i}}(1+x)^{s(A_i)}$. So the total generating function is just their sum.
\end{proof}

\begin{example}
	Let $\Gamma$ have the skeleton of the complete bipartite graph $K_{3,3}$ and let $K$ be the canonical divisor on $\Gamma$. If $L$ is the left divisor in Figure \ref{fig:div}, then $L$ has two chips on two edges respectively. Then  $2^{2-1}\cdot 2^{2-1}=4$ cells of $|K|$ are associated to $L$. By Theorem \ref{prop:dim}, their dimensions are $0,1,1,2$ respectively.
\end{example}

\subsection{Computing the anchor cells in $|D|$}\label{ssec:alg}

With the anchor cells, we have the following approach for finding the cells of $|D|$.

\begin{enumerate}[(a)]
	\item computing the anchor cells of $|D|$; \label{enu:anc}
	\item given the anchor cells of $|D|$, computing the other cells of $|D|$. \label{enu:asso}
\end{enumerate}

Step (\ref{enu:asso}) is done by Corollary \ref{cor:fvec}. We then explain implementations of Step (\ref{enu:anc}). We first introduce some properties of anchor cells.

\begin{lemma}\label{lem:2piece}
	Let $D$ be a vertex-supported effective divisor on a metric graph $\Gamma$. If $f\in R(D)$ and $D+(f)$ is an anchor divisor, then $f$ has at most two linear pieces on each edge of $\Gamma$.
\end{lemma}

\begin{proof}
	For any edge $e$ of $\Gamma$, suppose $x$ is an interior point of $e$ and it is the intersection of two linear pieces of $f$ with different slopes. Then, by definition, we have $(f)(x)\ne 0$. Since $D$ is vertex-supported, $D(x)=0$. Then $(D+(f))(x)\ne 0$. Since $f\in R(D)$, the divisor $D+(f)$ is effective, so $(D+(f))(x)>0$. However $D+(f)$ is an anchor divisor, meaning that there is at most one such point $x$, so $f$ has at most two linear pieces on $e$. 
	
	Therefore if the outgoing slopes at the two endpoints of $e$ sum to $0$, then $f$ is linear on $e$; otherwise $f$ has two linear pieces on $e$ and there is one interior point $x$ of $e$ such that $(D+(f))(x)>0$.
\end{proof}

\begin{corollary}\label{cor:anc}
	Let $D$ be a vertex-supported divisor on a metric graph $\Gamma$. If $C$ is an anchor cell and it is represented by a divisor $D+(f)$, then $C$ is uniquely determined by the outgoing slopes of $f$ at all vertices of $\Gamma$. 
\end{corollary}

\begin{proof}
	It suffices to show that given all those slopes, the data of $C$ are uniquely determined. Firstly, the $m_e$ are determined given those slopes. Secondly we show that the $d_e$ are also determined. If $C$ is an anchor cell and $D+(f)$ is a representative of $C$, then $D+(f)$ is an anchor divisor. By Lemma \ref{lem:2piece}, $f$ has at most two linear pieces on each edge $e$ of $\Gamma$. If $f$ is linear on $e$, then the integer $d_e$ is zero for $D+(f)$; otherwise $f$ has two linear pieces on $e$. Suppose they have a common point $v$. Then $(D+(f))(v)$ is minus the sum of the two outgoing slopes at the endpoints of $e$, so $d_e$ is still determined by those slopes. Finally for each vertex $v$ of $\Gamma$, $d_v=(D+(f))(v)$. $D(v)$ is known, and each $(f)(v)$, which is the sum of the outgoing slopes at v, is also determined, so $d_v$ is determined too.
\end{proof}

\begin{lemma}(\cite[Lemma 7]{HMY})\label{lem:bd}
	Let $D$ be an effective divisor on a metric graph $\Gamma$ and $f\in R(D)$. Then the slopes of all linear pieces of $f$ are between $-\deg(D)$ and $\deg(D)$.
\end{lemma}

By Corollary \ref{cor:anc} and Lemma \ref{lem:bd}, we can find all anchor cells by considering the $2|E|$-tuple of outgoing slopes of all $f\in R(D)$ with at most two linear piece(s) on each edge of $\Gamma$. In particular we have a proof of the finiteness of anchor cells.

Next we implement this approach using linear programming algorithms.  

Suppose $\Gamma=(V,E)$, where $V=\{v_1,\ldots,v_n\}$ is the set of vertices and $E=\{e_1,\ldots, e_m\}$ is the set of edges. For $1\le j\le m$, the edge $e_j$ has endpoints $v_{j(1)}$ and $v_{j(2)}$, where $1\le j(1),j(2)\le n$, and the length of $e_i$ is $M_i$. For $1\le i\le n$ let $d_i=D(v_i)$ and $d=\deg(D)=\sum_{i=1}^{n}{d_i}$.

\begin{lemma}\label{lem:ancfun}
	Let $f$ be a rational function defined on $\Gamma=(V,E)$ such that there are at most two linear pieces on each edge of $\Gamma$. Denote by $a_i=f(v_i)$ for $1\le i\le n$ and by $s_{j,1},s_{j,2}$ the outgoing slope at $v_{j(1)}$ and $v_{j(2)}$ of $e_j$ for $1\le j\le m$. Then $f\in R(D)$ if and only if
	\begin{itemize}
		\item for each vertex $v_i$ we have the equations
		\begin{equation}\label{eq:ver}
		d_i+\sum_{j(1)=i}{s_{j,1}}+\sum_{j(2)=i}{s_{j,2}}\ge 0;
		\end{equation}
		\item for $1\le j\le m$, either
		\begin{equation}\label{eq:1pc}
		s_{j,1}+s_{j,2}=0, a_{j(1)}-a_{j(2)}+s_{j,1}\cdot M_j=0	
		\end{equation}
		or
		\begin{equation}\label{eq:2pc}
		s_{j,1}+s_{j,2}<0, s_{j,2}\cdot M_j<a_{j(1)}-a_{j(2)}<-s_{j,1}\cdot M_j.
		\end{equation}
	\end{itemize}
\end{lemma}

\begin{proof}
	Suppose $f\in R(D)$. Note that for each $1\le i\le n$ we have $(f)(v_i)=\sum_{j(1)=i}{s_{j,1}}+\sum_{j(2)=i}{s_{j,2}}$, then
	\[d_i+\sum_{j(1)=i}{s_{j,1}}+\sum_{j(2)=i}{s_{j,2}}=(D+(f))(v_i)\ge 0.\]
	For each $1\le j\le m$, $f$ has either one or two linear piece(s) on the edge $e_j$. If there is one linear piece, then $s_{j,1}+s_{j,2}=0$, and by definition of slope we have $a_{j(2)}-a_{j(1)}=s_{j,1}\cdot M_j$; if there are two linear pieces, then there exists an interior point $p$ of $e_j$ such that $f$ is linear on both line segments $v_{j(1)}p$ and $v_{j(2)}p$, and $s_{j,1}+s_{j,2}\ne 0$. Note that $(D+(f))(p)=(f)(p)=-s_{j,1}-s_{j,2}$, so $s_{j,1}+s_{j,2}<0$. Let $x$ be the length of the segment $v_{j(1)}p$, then the length of the segment $v_{j(2)}p$ is $M_j-x$. And by definition of slope we have
	\begin{equation}\label{eq:slo}
	a_{j(1)}-f(p)+s_{j,1}\cdot x=0, \quad a_{j(2)}-f(p)+s_{j,2}\cdot (M_j-x)=0.
	\end{equation}
	Since $f(p)$ can be any real number, we eliminate it and get
	\[(s_{j,1}+s_{j,2})x=(a_{j(2)}-a_{j(1)}+s_{j,2}M_j).\]
	Then 
	\[x=\frac{a_{j(2)}-a_{j(1)}+s_{j,2}\cdot M_j}{s_{j,1}+s_{j,2}}.\]
	Since $p$ is an anterior point of $e_j$, we have $0<x<M_j$. Then $a_{j(2)}-a_{j(1)}+s_{j,2}\cdot M_j<0$ and $a_{j(1)}-a_{j(2)}+s_{j,1}\cdot M_j<0$.
	
	Conversely, suppose for each $1\le i\le n$ (\ref{eq:ver}) holds and for each $1\le j\le m$, either (\ref{eq:1pc}) or (\ref{eq:2pc}) holds. For each $1\le j\le m$, if $s_{j,1}+s_{j,2}=0$ then by (\ref{eq:1pc}) $f$ is well-defined on $e_j$ and $D+(f)(q)\ge 0$ for all points $q$ in the interior of $e_j$; otherwise (\ref{eq:2pc}) holds, by the chain inequalities $f$ is well-defined on $e_j$ with two linear pieces. Since $s_{j,1}+s_{j,2}<0$, we have $D+(f)(q)\ge 0$ for all points $q$ in the interior of $e_j$. Finally for each $1\le i\le n$, (\ref{eq:ver}) holds. Then $(D+(f))(v_i)\ge 0$. So $f$ is a well-defined rational function on $\Gamma$ with at most two linear pieces on each edge of $\Gamma$ and $f\in R(D)$.
\end{proof}

Note that if $D+(f)$ is an anchor divisor in $|D|$, then it also has degree $d$ and its support has at most one intersection point with each open $e_i$. Thus we obtain a partition of $d$ into $n+m$ nonnegative integers: $d=\sum_{i=1}^{n}{d'_i}+\sum_{j=1}^{m}{c_j}$, where $d'_i=(D+(f))(v_i)$ and $c_j=\sum_{x\in {e_j}^{\circ}}{(D+(f))(x)}$. We call them as \emph{configurations} of anchor divisors. Note that there are $\binom{d+n+m-1}{d}$ configurations in total.

\begin{corollary}\label{cor:cfg}
	Suppose $f\in R(D)$ with the same notations as in Lemma \ref{lem:ancfun}. Let $d=\sum_{i=1}^{n}{d'_i}+\sum_{j=1}^{m}{c_j}$ be the configuration of the divisor $D+(f)$, then for $1\le j\le m$, we have
	\begin{equation}\label{eq:cfge}
		c_j=-s_{j,1}-s_{j,2}.
	\end{equation}
	And for $1\le i\le n$ we have 
	\begin{equation}\label{eq:cfgv}
		d'_i=d_i+\sum_{j(1)=i}{s_{j,1}}+\sum_{j(2)=i}{s_{j,2}}
	\end{equation}
\end{corollary}

Now for each configuration $d=\sum_{i=1}^{n}{d'_i}+\sum_{j=1}^{m}{c_j}$, we consider the following system of linear constraints in  (\ref{eq:1pc}), (\ref{eq:2pc}), (\ref{eq:cfge}),  (\ref{eq:cfgv}). Here $d_i,d'_i,M_j,c_j$ are parameters and $a_i,s_{j,1},s_{j,2}$ are variables. Then among the solutions of this system, each $2m$-tuple of integers $s_{j,1},s_{j,2}$ gives an anchor cell of $|D|$.

\begin{remark}
	For each $j$, whether to apply (\ref{eq:1pc}) or (\ref{eq:2pc}) depends on the value of $c_j$.
\end{remark}

Lemma \ref{lem:ancfun} and Corollary \ref{cor:cfg} gives rise to the following algorithm. The input is the metric graph $\Gamma=(V,E)$, the edge lengths $M_j$ and the divisor $D$. The output $L$ is a list of the anchor cells.

\begin{algorithm}
	\caption{Computing Anchor Cells}\label{alg:anc}
	\begin{algorithmic}[1]
		\Procedure{AnchorCells}{$\Gamma=(V,E),M,D$}
		\State $L\gets \emptyset$
		\State $m\gets |E|$
		\State $n\gets |V|$
		\State $d\gets \deg(D)$
		\State $s\gets \{c=(c_1,\ldots,c_m,d'_1,d'_2,\ldots,d'_n)|c_j,d'_i\in \mathbb{N}, \sum_{j=1}^{m}{c_j}+\sum_{i=1}^{n}{d'_i}=d.\}$ 
		\For{$c\in s$}
		\State $S \gets$ the system of linear constraints in (\ref{eq:1pc}), (\ref{eq:2pc}), (\ref{eq:cfge}), (\ref{eq:cfgv}) 
		\If{$S$ does not have a solution with all $s_{j,1},s_{j,2}$ being integers} \hfill (*)
			\State next $c$
		\EndIf
			\State $a\gets $ array of $m$ entries
		    \State $b\gets $ array of $m$ entries
			\For{$j\gets 1,m$}
				\State $a[j] \gets [minimize(s_{j,1},S),maximize(s_{j,1},S)] \cap \mathbb{Z}$
				\State $b[j] \gets [minimize(s_{j,2},S),maximize(s_{j,2},S)] \cap \mathbb{Z}$
			\EndFor
			\State $T\gets $ the Cartesian product of $a[1],\ldots,a[m],b[1],\ldots,b[m]$
			\For{$v\in T$}
				\If{$S$ union $\{s_{j,1}=v_j,s_{j,2}=v_{j+m}|j=1,\ldots,m\}$ is feasible}
					 \State $L\gets L \cup \{[v,c]\}$
				\EndIf
			\EndFor
		\EndFor
		\State \textbf{return} $L$
		\EndProcedure
	\end{algorithmic}
\end{algorithm}

\begin{remark}
	(1) We use integer programming method for step (*). In {\tt Maple 2015} there is a command {\tt LPSolve} that is able to do it.\\
	  
	(2) In this algorithm we do not require the metric graph to be simple. We allow both loops and parallel edges.
\end{remark}

We can also deal with more general input. Suppose the skeleton of $\Gamma$ is fixed, but the metric may vary. In this case the input is the metric $(M_j)_{1\le j\le m}$, and the desired output is the set of anchor cells in $|D|$. Apparently we can apply the above approach once the $M_j$ are given, but since Algorithm \ref{alg:anc} needs to run the linear-programming subroutine $\binom{d+n+m-1}{d}$ times, it is not very efficient if we would like to compute for many different metrics.

Instead, we can also view the $M_j$'s as variables. However $s_{j,1}\cdot M_j$ appears in (\ref{eq:2pc}). In order to make the system of constraints linear, we have to let $s_{j,1},s_{j,2}$ become parameters. So we take the approach in the proof of Proposition \ref{thm:fin}. 
We still have the equations and inequalities (\ref{eq:1pc}), (\ref{eq:2pc}), (\ref{eq:cfge}), (\ref{eq:cfgv}), but $s_{j,1},s_{j,2}$ are parameters instead. Now the parameters are $d_i,s_{j,1},s_{j,2}$ and the variables are $M_j,d'_i,a_i,c_j$. Then we also need
\begin{equation}\label{ein:eff}
    c_j\ge 0 \quad \forall 1\le j\le m \text{ and } d'_i\ge 0 \quad \forall 1\le i\le n
\end{equation}
for $D+(f)$ being effective and
\begin{equation}\label{ein:pos}
    M_j>0 \quad \forall 1\le j\le m.
\end{equation}
for the edge lengths in $M$ are all positive.

By Lemma \ref{lem:bd}, we have $|s_{j,1}|,|s_{j,2}|\le d$. So there are finitely many possible values of them. Now we have an empty list first and for each choice of $(s_{j,k})$, we check the feasibility of the system  of linear constraints formed by (\ref{eq:1pc}), (\ref{eq:2pc}), (\ref{eq:cfge}), (\ref{eq:cfgv}), (\ref{ein:eff}), (\ref{ein:pos}). If it is feasible, then we find one anchor cell and save it to our list. The output is the set of anchor cells represented by divisors $L$ such that there exists some metric $M=(M_j)_{1\le j\le m}$ with $L\in |D|_{M}$. 

The advantage of this approach is that we can compute this list of all possible anchor cells in $|D|$ beforehand. Then given a specific metric, we just plug in the values of $M_j$ and check the feasibility of each anchor cell in the list.

\section{Chip-firing and extremal generators of $R(D)$}\label{sec:R(D)}

In this section we present the properties of $R(D)$. If $f$ and $g$ are rational functions defined on $\Gamma$, then $f\oplus g$ is the rational function on $\Gamma$ with $(f\oplus g)(x)=\max(f(x),g(x))$ and $f\odot g$ is the rational function on $\Gamma$ with $(f\odot g)(x)=f(x)+g(x)$. In other words, $\oplus$ and $\odot$ are the \emph{tropical} operations in the set of rational functions on $\Gamma$. Here we choose the max-plus algebra.

\begin{lemma}(\cite[Lemma 4]{HMY})
	Let $D$ be any divisor on a metric graph $\Gamma$. The space $R(D)$ is a tropical semi-module, i.e. it is closed under tropical addition and tropical scalar multiplication.
\end{lemma}

We would like to find a minimal set of generators of the tropical semi-module $R(D)$. We use the notion of \emph{chip-firing} \cite{BLS}.

For an effective divisor $D$ on $\Gamma$ we regard it as a distribution of $\deg(D)$ chips on $\Gamma$: there are $D(x)$ chips at each point $x\in \Gamma$. Suppose that $D$ is effective and $f\in R(D)$. For each linear piece $\overline{PQ}$ of $\Gamma$, if $f$ has slope $s\in \Z_{-}$ from $P$ to $Q$, then we say that $P$ fires $s$ chips to $Q$ when adding $(f)$ to $D$. 

For a metric graph $\Gamma$, a \emph{subgraph} is a compact subset with a finite number of components. Fix an effective divisor $D$ on $\Gamma$. We say that a subgraph $\Gamma'$ of $\Gamma$ can \emph{fire} for $D$ if, for each boundary point $x$ of $\Gamma' \cap \overline{\Gamma-\Gamma'}$, the number of edges pointing out of $\Gamma'$ is no greater than $D(x)$. 

A function $f\in R(D)$ is called \emph{extremal} if for any $g_1,g_2 \in R(D)$, the decomposition $f=g_{1}\oplus g_{2}$ implies that $f=g_1 \text{ or }f=g_2$. Any generating set of $R(D)$ must contain all extremal generators up to tropical scalar multiplication.

\begin{lemma}(\cite[Theorem 14(a)]{HMY})\label{lem:ext}
	If $D$ is a vertex-supported divisor on a metric graph $\Gamma$ and $f\in R(D)$ is extremal, then $D+(f)$ is a vertex of the cell complex $|D|$. 
\end{lemma}

By Lemma \ref{lem:ext} in order to find the extremal generators of $R(D)$, it suffices to search among the vertices of $|D|$. The next lemma is an important criterion for the extremal generators of $R(D)$.

\begin{lemma}(\cite[Lemma 5]{HMY})\label{lem:fire}
	Let $D$ be any divisor on a metric graph $\Gamma$. Then $f\in R(D)$ is extremal if and only if there do not exist two proper subgraphs $\Gamma_1$ and $\Gamma_2$ of $\Gamma$ such that they cover $\Gamma$ and both can fire on $D+(f)$.
\end{lemma}

\begin{remark}
	For a non-extremal $f\in R(D)$, the proper subgraphs $\Gamma_1$ and $\Gamma_2$ may not be obvious. See the example in Figure \ref{fig:div}, where $\Gamma$ is the metric graph with skeleton $K_{3,3}$ and all-equal metric, $D=K$ is the canonical divisor on $\Gamma$.
\end{remark}

\begin{figure}[h]
	\color{black}
	\centering
	\begin{minipage}[t]{0.1\textwidth}
		\centering
		\begin{tikzpicture}[scale=0.3]
		\draw (0,0) -- (0,6) -- (3,0) --(6,6) -- (6,0) -- (3,6) -- (0,0) --(6,6);
		\draw (3,6) -- (3,0);
		\draw (0,6) -- (6,0);
		\filldraw [black] (0,3) circle (3pt);
		\filldraw [black] (3,0) circle (3pt);
		\filldraw [black] (3,6) circle (3pt);
		\filldraw [black] (6,3) circle (3pt);
		\node [left] at (0,3) {2};
		\node [right] at (6,3) {2};
		\node [below] at (3,0) {1};
		\node [above] at (3,6) {1};			
		\end{tikzpicture}
	\end{minipage}
	\hspace{2cm}
	\begin{minipage}[t]{0.1\textwidth}
		\centering
		\begin{tikzpicture}[scale=0.3]
		\draw [red] (0,3) -- (0,6) -- (3,0) --(3,6) -- (6,0) -- (6,3);
		\draw [red] (0,6) -- (6,0);
		\draw [blue,thick] (0,3) -- (0,1.8);
		\draw [blue,thick] (6,3) -- (6,4.2);
		\draw [blue,thick] (3,0) -- (3.5,1);
		\draw [blue,thick] (3,6) -- (2.5,5);
		\draw [black] (0,1.8) -- (0,0) -- (2.5,5);
		\draw [black] (6,4.2) -- (6,6) -- (3.5,1);
		\draw [black] (0,0) -- (6,6);				
		\filldraw [black] (0,3) circle (3pt);
		\filldraw [black] (3,0) circle (3pt);
		\filldraw [black] (3,6) circle (3pt);
		\filldraw [black] (6,3) circle (3pt);
		\node [left]  at (0,3) {2};
		\node [right] at (6,3) {2};
		\node [below] at (3,0) {1};
		\node [above] at (3,6) {1};			
		\end{tikzpicture}
	\end{minipage}
	\hspace{2cm}
	\begin{minipage}[t]{0.1\textwidth}
		\centering
		\begin{tikzpicture}[scale=0.3]
		\draw [red] (0,3) -- (0,0) -- (3,6) --(3,0) -- (6,6) -- (6,3);
		\draw [red] (0,0) -- (6,6);
		\draw [blue,thick] (0,3) -- (0,4.2);
		\draw [blue,thick] (6,3) -- (6,1.8);
		\draw [blue,thick] (3,0) -- (2.5,1);
		\draw [blue,thick] (3,6) -- (3.5,5);
		\draw [black] (0,4.2) -- (0,6) -- (2.5,1);
		\draw [black] (6,1.8) -- (6,0) -- (3.5,5);
		\draw [black] (0,6) -- (6,0);		
		\filldraw [black] (0,3) circle (3pt);
		\filldraw [black] (3,0) circle (3pt);
		\filldraw [black] (3,6) circle (3pt);
		\filldraw [black] (6,3) circle (3pt);
		\node [left]  at (0,3) {2};
		\node [right] at (6,3) {2};
		\node [below] at (3,0) {1};
		\node [above] at (3,6) {1};			
		\end{tikzpicture}
	\end{minipage}	
	\caption{A non-extremal divisor and the two subgraphs (blue) that can fire. The corresponding rational function takes value $1$ on the red parts and $0$ on the black parts and is linear with slope $1$ from red parts to black parts.}\label{fig:div}
\end{figure}
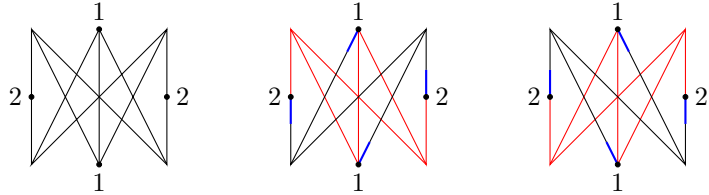

Our discussion suggests the following approach for computing all extremal generators of $R(D)$: 
\begin{enumerate}[(A)]
	\item Find all anchor divisors in $R(D)$.
	\item Filter them by Theorem \ref{prop:dim} to get all vertices of $|D|$.
	\item Filter the vertices by Lemma \ref{lem:fire} to get all extremal generators of $R(D)$.\label{enu:ext}
\end{enumerate}

We introduce a method for Step (\ref{enu:ext}). Given an effective divisor $L=D+(f)$ on $\Gamma$, we apply Lemma \ref{lem:fire} to check whether $f$ is extremal. Note that if a proper subgraph can fire for $L$, then its boundary is contained in $supp(L)$. So we partition $\Gamma$ into connected subgraphs whose boundaries are contained in $supp(L)$ and we call them components. Since $supp(L)$ is finite, so is the number of such components. We conclude that any subgraph that can fire for $L$ is a union of these components. For each such union, we ignore it if it is non-proper or cannot fire for $L$, then we have a finite list of all proper subgraphs of $\Gamma$ that can fire for $L$. Finally we check whether there is a pair in the list that covers $\Gamma$.

\section{Canonical linear systems on some trivalent graphs}\label{sec:exp}

In this section we apply our methods to some nontrivial examples. A metric graph is \emph{trivalent} if the degree of every vertex is $3$. Trivalent graphs appear in the Berkovich skeleton of many curves\cite[Example 5.29]{BPR}. This fact motivates us to compute the examples of trivalent metric graphs and the canonical divisor $K$. 

For the computations we performed below, our hardware is a laptop with Intel Core i5-6200U processor (2.3 GHz) and 8GB RAM. The software is {\tt Maple}, version 2015. All computations were single-threaded.

\subsection{The example $K_{4}$}\label{ssec:k4}

The complete graph $K_4$ has $4$ vertices and its genus is $3$. Table \ref{tab:K4} shows the structure of $R_{M,K}$ and $|K|_M$ on $K_{4}$ given a particular metric $M$. For $1\le i<j\le 4$ let $l_{ij}$ be the edge length between vertex $i$ and $j$. A metric $M$ is given by $(l_{12},l_{13},l_{14},l_{23},l_{24},l_{34})$. 

\begin{center}
	\begin{tabular}{|c|c|c|c|c|}
	    \toprule
	    \textbf{Metric} & \tabincell{c}{\textbf{Anchor} \\ \textbf{Cells}}& \tabincell{c}{\textbf{Extremal} \\ \textbf{Generators}} &  \textbf{$f$-vector} & \tabincell{c}{\textbf{Computational} \\ \textbf{Time (s)}}\\
	    \hline 
	    $(1,1,1,1,1,1)$ & $30$ & $7$ & $(14,28,15)$ & $3.1$ \\
	    \hline
	    $(1,1,2,2,1,1)$ & $42$ & $11$ & $(26,52,31,4)$ & $3.4$ \\
	    \hline
	    $(2,2,2,2,2,3)$ & $36$ & $9$ & $(20,40,23,2)$ & $3.0$ \\
	    \hline
	    $(2,2,2,2,2,1)$ & $40$ & $11$ & $(24,44,21)$ & $3.8$ \\
	    \hline
	    $(4,9,7,8,6,10)$ & $50$ & $15$ & $(34,60,27)$ & $4.1$ \\
	    \bottomrule
	\end{tabular}
	\captionof{table}{Structure of $|K|$ and $R_{K}$ for different metrics on $K_{4}$}\label{tab:K4}
\end{center}

\begin{conjecture}\label{conj:16}
	If $\Gamma$ has skeleton $K_4$ and $D=K$, then the number of anchor cells in $|D|$ minus the number of vertices in $|D|$ is always $16$.
\end{conjecture}

\subsection{The example of $(020)$}\label{ssec:020}

Following the notation of \cite{BJMS}, we denote by $(020)$ the following trivalent graph with $4$ vertices and $6$ edges (Figure \ref{fig:020}). The metric is the vector $(a_1,a_2,b,c,d_1,d_2)$. 

\begin{figure}[H]
	\centering
	\includegraphics[scale=0.5]{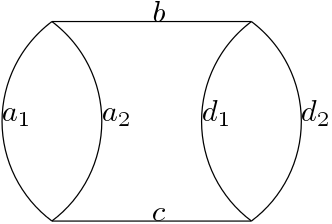}
	\caption{The trivalent graph $(020)$}\label{fig:020}
\end{figure}

Table \ref{tab:020} shows the structure of $R_{M,K}$ and $|K|_M$ on $(020)$ for some metric $M$.

\begin{center}
	\begin{tabular}{|c|c|c|c|c|}
		\toprule
		\textbf{Metric} & \tabincell{c}{\textbf{Anchor} \\ \textbf{Cells}}& \tabincell{c}{\textbf{Extremal} \\ \textbf{Generators}} &  \textbf{$f$-vector} & \tabincell{c}{\textbf{Computational} \\ \textbf{Time (s)}}\\
		\hline 
		$(1,1,1,1,1,1)$ & $20$ & $42$ & $(31,61,36,5)$ & $3.1$ \\
		\hline
		$(1,1,1,2,1,1)$ & $12$ & $44$ & $(25,47,24,1)$ & $3.5$ \\
		\hline
		$(1,3,2,2,1,3)$ & $20$ & $42$ & $(31,61,36,5)$ & $3.0$ \\
		\bottomrule
	\end{tabular}
	\captionof{table}{Structure of $|K|$ and $R_{K}$ for different metrics on $(020)$}\label{tab:020}
\end{center}

\subsection{The example $K_{3,3}$}\label{ssec:k33}

Table \ref{tab:K33} shows the structure of $R_{M,K}$ and $|K|_M$ on $K_{3,3}$ given a particular metric $M$. 
\begin{center}
	\begin{tabular}{|c|c|c|c|c|}
		\toprule
			    \textbf{Metric} & \tabincell{c}{\textbf{Anchor} \\ \textbf{Cells}}& \tabincell{c}{\textbf{Extremal} \\ \textbf{Generators}} &  \textbf{$f$-vector} & \tabincell{c}{\textbf{Computational} \\ \textbf{Time (s)}} \\
	    \hline
	    All-equal & 370 & 33 & $(130,483,630,348,81,9)$ & $171.9$ \\
		\hline
		{\small $\begin{bmatrix}
			2 & 1 & 1\\
			1 & 2 & 1\\
			1 & 1 & 2
			\end{bmatrix}$} & 460 & 63 & $(196,615,666,276,33,3)$ & $189.8$ \\	    	    		
		\hline
		{\small $\begin{bmatrix}
			3 & 91 & 96\\
			94 & 4 & 92\\
			93 & 95 & 5
			\end{bmatrix}$} & 730 & 84 & $(337,936,873,273)$ & $241.9$ \\
		\bottomrule
	\end{tabular}
	\captionof{table}{Structure of $R_{M,k}$ and $|K|$ for different metrics on $K_{3,3}$}\label{tab:K33}
\end{center}

\begin{remark}\label{rem:char}
	By \cite[Corollary 31]{HMY}, the cell complex $|D|$ is contractible as a topological space. Thus the Euler characteristic of $|D|$ is always $1$, which coincides with all the $f$-vectors we computed above.
\end{remark}

\section{Further research}\label{sec:q}

Here are some open problems for further research on this topic. 
\begin{itemize}
	\item Fixing the skeleton of $\Gamma$ and $D$, find the polyhedral cone decomposition of $\mathbb{R}_{>0}^{|E|}$ based on the cell structure of $|D|$. In particular, determine the metrics $M$ such that $|D|_{M}$ has dimension $\deg(D)-1$. 
	\item Given all the cells in $|D|$, find the face lattice of $|D|$. In general, if we apply chip-firing to a representative of a cell, we may obtain a representative of a cell on its boundary, but the rigorous argument is yet to be established.
	\item Given the skeleton of $\Gamma$ and $D$, find non-trivial upper and lower bounds of the number of anchor cells (or cells, vertices) in $|D|$ and of the number of extremal generators in $R(D)$.
\end{itemize}

\bibliographystyle{abbrv}

\nocite{MS}
\section*{Acknowledgements}
The author is grateful to Bernd Sturmfels for his guidance and encouragement throughout this project. The author thanks Madhusudan Manjunath, Ralph Morrison and Kristian Ranestad for helpful discussions and comments. The author thanks Yang An, Matt Baker and Josephine Yu for their precious suggestions. The author also thanks two anonymous referees for their critical comments on an earlier version of this work.

\bigskip

\noindent
\footnotesize {\bf Authors' address:}

\noindent Bo Lin, University of California, Berkeley, USA, 94720
{\tt linbo@math.berkeley.edu}

\end{document}